\documentclass[12pt,reqno]{amsart}
\usepackage{amsmath,amssymb,amsfonts,amscd,latexsym,amsthm,mathrsfs,verbatim,comment}
\usepackage[unicode]{hyperref}
\usepackage{hypbmsec}
\newtheorem{theorem}{Theorem}
\newtheorem{lemma}{Lemma}
\theoremstyle{definition}

\newtheorem{conjecture}{Conjecture} 
\newtheorem*{acknowledgements}{Acknowledgements}

\renewcommand{\d}{{\mathrm d}}
\newcommand{\agm}{\operatorname{agm}}

\newcommand{\bS}{\mathbf{S}}
\begin{document}

\title{A magnetic double integral}

\date{27 March 2018}

\author{David Broadhurst}
\address{School of Physical Sciences, Open University, Milton Keynes MK7\,6AA, UK}
\email{david.broadhurst@open.ac.uk}

\author{Wadim Zudilin}
\address{IMAPP, Radboud Universiteit, PO Box 9010, 6500 GL Nijmegen, The Netherlands}
\email{w.zudilin@math.ru.nl}
\address{MAPS, The University of Newcastle, Callaghan, NSW 2308, Australia}
\email{wadim.zudilin@newcastle.edu.au}

\dedicatory{To the memory of Jonathan Borwein (1951--2016)}


\begin{abstract}
In a recent study of how the output voltage of a Hall plate is affected by the shape of the plate and the size of its contacts,
U.~Ausserlechner has come up with a remarkable double integral that can be viewed as a generalization of the classical elliptic ``AGM'' integral.
Here we discuss transformation properties of the integral, which were experimentally observed by Ausserlechner, as well as its analytical
and arithmetic features including connections with modular forms.
\end{abstract}

\maketitle

\section{Introduction}
\label{sec1}

Recall that the arithmetic-geometric mean (AGM) $\agm(a,b)$ of two positive real numbers $a$ and $b$
is defined as a common limit of the sequences $a_0=a$, $a_1,a_2,\dots$ and $b_0=b$, $b_1,b_2,\dots$ generated by the iteration
$$
a_{n+1}=\frac{a_n+b_n}2 \quad\text{and}\quad b_{n+1}=\sqrt{a_nb_n}
\qquad\text{for}\quad n=0,1,2,\dotsc.
$$
Thanks to Gauss,
$$
\frac{\pi}{2\agm(a,b)}=\int_0^{\pi/2}\frac{\d\alpha}{\sqrt{a^2\cos^2\alpha+b^2\sin^2\alpha}}
=\frac{I_1(b/a)}{a},
$$
where
$$
I_1(f)=\int_0^{\pi/2}\frac{\d\alpha}{\sqrt{\cos^2\alpha+f^2\sin^2\alpha}}
$$
is a particular instance of a classical elliptic integral. The AGM sources several beautiful structures, formulae and algorithms in mathematics;
it is linked with periods of elliptic curves, modular forms and the hypergeometric function. A standard reference to all these features of the AGM
is the book \cite{BB87} by the Borweins.

In a recent study of how the output voltage of a Hall plate is affected by the shape of the plate and the size of its contacts,
U.~Ausserlechner \cite{Au16} has come up with the double integral
\begin{equation}
I_2(f)=(1+f)\iint_{\pi/2>\alpha>\beta>0}\frac{\d\alpha\,\d\beta}
{\sqrt{((1+f)^2-4f\cos^2\alpha)((1+f)^2-4f\sin^2\beta)}}
\label{dint}
\end{equation}
and conjectured that
\begin{equation}
I_2(f)=I_2\biggl(\frac{1-f}{1+f}\biggr)
\qquad\text{for}\quad f\in[0,1].
\label{udo}
\end{equation}
In this note we infer a differential equation for $I_2(f)$ that implies, among other things, the involution \eqref{udo}.

During the preparation of this work for publication we learnt of the proof of \eqref{udo} by L.~Glasser and Y.~Zhou \cite{GZ17}.
Though the method in \cite{GZ17} is seemingly less laborious than the strategy below, our approach is different and reveals further structure of $I_2(f)$.

\section{Preliminaries}
\label{sec:pre}

First observe that
$$
I_2(-f)=(1-f)\iint_{\pi/2>\alpha>\beta>0}\frac{\d\alpha\,\d\beta}
{\sqrt{((1+f)^2-4f\sin^2\alpha)((1+f)^2-4f\cos^2\beta)}}
$$
and hence the even function
\begin{equation}
J(f)=\frac{I_2(f)}{1+f}+\frac{I_2(-f)}{1-f}=\biggl(\frac{\pi/2}{\agm(1+f,1-f)}\biggr)^2
\label{agm}
\end{equation}
is given by the square of an AGM.

Now let $\theta=\theta_f=f\,\frac{\d}{\d f}$ and
$$
L(f,\theta)=\frac1f\,(\theta^3-2f(\theta^3+\theta)f+f^2\theta^3f^2)\,\frac1{1+f}.
$$
Then from the linear differential equation satisfied by $J(f)$, we prove that
\begin{equation}
R(f)=L\biggl(f,f\,\frac{\d}{\d f}\biggr)I_2(f)
\label{RviaL}
\end{equation}
is an even function of $f$. From study of the expansion of $I_2(f)$ about $f=0$, we infer the following.

\begin{theorem}
\label{th1}
We have
\begin{equation}
R(f)=2\biggl(\frac{2f}{1+f^2}\biggr)^2-1.
\label{rel}
\end{equation}
\end{theorem}

In spite of its technical nature, our proof of the theorem follows a standard route.
The discovery of \eqref{rel} and its relevance to \eqref{udo} should be counted as the principal ingredients of investigation in this note.
We postpone the details of proof to Section~\ref{sec:DE}.

\begin{lemma}
\label{lem1}
Relation \eqref{rel} implies the transformation in \eqref{udo}.
\end{lemma}

\begin{proof}
Let $g=(1-f)/(1+f)$. Then the relation $(1+f)(1+g)=2$ translates into $L(f,f\,\d/\d f)+L(g,g\,\d/\d g)=0$. Moreover, \eqref{rel}
gives $R(f)+R(g)=0$. Thus $I_2(f)$ and $I_2(g)$ satisfy the same inhomogeneous differential equation, which has only one solution that is regular about $f=0$.
Hence we only need to demonstrate that $I_2(0)=I_2(1)$ to conclude that \eqref{rel} implies \eqref{udo}.

The evaluation
$$
I_2(0)=\int_0^{\pi/2}\d\alpha\int_0^\alpha\d\beta=\frac{\pi^2}8
$$
is elementary. For
$$
I_2(1)=\frac12\int_0^{\pi/2}\frac{\d\alpha}{\sin\alpha}\int_0^\alpha\frac{\d\beta}{\cos\beta}
$$
we expand the inner integral in terms of $\sin\alpha$ and obtain
$$
I_2(1)=\frac12\sum_{n=0}^\infty\int_0^{\pi/2}\frac{\sin^{2n}\alpha}{2n+1}\,\d\alpha
=\frac\pi4\sum_{n=0}^\infty\frac{(\frac12)_n^2}{(\frac32)_n\,n!}=\frac{\pi^2}8,
$$
and so we are done.
\end{proof}

Here and in the latter proof of Theorem~\ref{th1} we invoked the standard notation
$$
(\alpha)_n=\frac{\Gamma(\alpha+n)}{\Gamma(\alpha)}=\begin{cases}
\alpha(\alpha+1)\dotsb(\alpha+n-1) &\text{for $n\ge1$}, \\
1 &\text{for $n=0$},
\end{cases}
$$
for the Pochhammer symbol.

\section{Numerical evaluation of the double integral}
\label{sec:NE}

Ausserlechner provided \cite{Au16} an expression for \eqref{dint} as a single integral of a complete elliptic integral, thereby reducing
the burden of numerical computation of the double integral. Here we give an efficient numerical method that follows from the differential equation that we have exposed.

For $f<0$ and $f\ne-1$, we use \eqref{agm} to map to the region $f>0$.
For $f>1$, we use $I_2(f)=I_2(1/f)/f$ to map to $f\in[0,1]$.
For $f\in[\sqrt2-1,1]$, we use \eqref{udo} to map to $f\in[0,\sqrt2-1]$.
Thus we now have $f^2\le(\sqrt2-1)^2<0.172$, irrespective of where we began on the real line.
Hence the series expansion for the odd terms in
$$
\frac{I_2(f)}{1+f}
=\frac{\pi^2}8\biggl(\frac1{\agm(1+f,1-f)}\biggr)^2-\sum_{n=0}^\infty a_nf^{2n+1}
$$
converges rapidly. The expansion coefficients $a_n$ are efficiently
delivered by a recursion that follows from the differential equation \eqref{RviaL} with the left-hand side given by~\eqref{rel}.
We set $a_{-1}=0$ and $a_0=1$. Then for $n>0$ we use the recursion
$$
(2n+1)^3a_n=4n(4n^2+1)a_{n-1}-(2n-1)^3a_{n-2}+8(-1)^nn
$$
with an inhomogeneous final term from \eqref{rel}.

At the CM point $f_c=\sqrt2-1$, which is the fixed point of transformation \eqref{udo}, 
we obtain $2f_c\,I_2(f_c)=I_2(-f_c)$  by setting $f=f_c$, $\lambda=\sqrt2$ in \cite[Eq.~(53a)]{Au16}
and using the geometrical Hall factor $G_{\textup{H0}} = 2/3$ attributed to Haeusler \cite{Ha66} by
Ausserlechner \cite[above Eq.~(54c)]{Au16}. Then from our Eq.~\eqref{agm} we obtain
$$
I_2(\sqrt2-1)=\frac\pi{48\sqrt2}\biggl(\frac{\Gamma(\frac18)}{\Gamma(\frac58)}\biggr)^2
$$
by reducing the square of an elliptic integral at the second singular value
(see \cite[Section 1.6, Exercise 4]{BB87}) to a quotient of gamma values that results   
from the Chowla--Selberg formula (set $d = -8$, $w  = 2$, $h = 1$ in \cite[Eq.~(7)]{CS49}).

The modular parametrisation of the AGM and the inhomogeneous differential equation~\eqref{RviaL},~\eqref{rel} give rise to
modularity properties of the integral $I_2(f)$. We transform from $f$ to the nome $q$ defined by
$$
q=\exp\biggl(-\frac{\pi\agm(1+f,1-f)}{\agm(1,f)}\biggr),
$$
which gives
\begin{equation}
f^2=16\biggl(\frac{\eta_1\eta_4^2}{\eta_2^3}\biggr)^8
=1-\biggl(\frac{\eta_1^2\eta_4}{\eta_2^3}\biggr)^8
\label{deff}
\end{equation}
with
$$
\eta_m=q^{m/24}\prod_{k=1}^\infty(1-q^{mk}).
$$
Then the third-order inhomogeneous equation has the form
$$
\biggl(q\,\frac{\d}{\d q}\biggr)^3
\biggl(\frac{(\agm(1+f,1-f))^2I_2(f)}{1+f}\biggr)
=\phi(\tau)
$$
with $q=\exp(2\pi i\tau)$. To transform the result of Theorem~\ref{th1}, we use
$$
\biggl(q\,\frac{\d}{\d q}\biggr)\log(f^2)=\biggl(\frac{\eta_1^2}{\eta_2}\biggr)^4
$$
and obtain the inhomogeneous term as
\begin{equation}
\phi(\tau)=\frac12(\eta_1\eta_2)^4R(f)
=-\sum_{n=0}^\infty\Bigl(n+\frac12\Bigr)A(n)q^{n+1/2}
\label{phi}
\end{equation}
with integer values of $A(n)$ for $n\leq32200$, beginning with
\begin{gather*}
1, \, -44, \, 1126, \, -27096, \, 640909, \, -15036548, \, 351245038, \\ -8183857544, \, 190367634194, \, -4423279591132, \, \dots
\end{gather*}
for $n=0,1,\dots,9,\dots$\,. We expect the following to be true.

\begin{conjecture}
\label{conj1}
All the numbers $A(n)$ occurring in the Fourier expansion \eqref{phi} of the weight $4$ modular form are integral.
\end{conjecture}

Three integrations give
$$
\frac{(\agm(1+f,1-f))^2I_2(f)}{1+f}
=\frac{\pi^2}{8}-\sum_{n=0}^\infty\frac{A(n)}{(n+\frac12)^2}\,q^{n+1/2}.
$$
At $f=\sqrt2-1$, we have $q=\exp(-\sqrt2\pi)$ and
$$
\sum_{n=0}^\infty\frac{A(n)}{(n+\frac12)^2}\,\exp\biggl(-\frac{(2n+1)\pi}{\sqrt2}\biggr)
=\frac{\pi^2}{24}.
$$

\section({Properties of the modular form \003\306(\003\304)})%
{Properties of the modular form $\phi(\tau)$}
\label{sec:phi}

\begin{lemma}
\label{lem2}
The weight $4$ modular form in \eqref{phi} satisfies
$$
\phi(\tau+1)=-\phi(\tau)=\frac{1}{4\tau^4}\,\phi\biggl(\frac{-1}{2\tau}\biggr).
$$
\end{lemma}

\begin{proof}
The transformation $\tau\mapsto\tau+1$ leaves $R(f)$ unchanged and brings a minus sign from $(\eta_1\eta_2)^4$.
The involution $\tau\mapsto-1/(2\tau)$ brings a factor $1/(4\tau^4)$ from $(\eta_1\eta_2)^4$ and a minus sign from $R(f)$.
To show the origin of the latter sign, consider
$$
\psi(\tau)=\biggl(\frac{2f}{1-f^2}\biggr)^2=64\biggl(\frac{\eta_2}{\eta_1}\biggr)^{24}
$$
determined as an eta quotient by~\eqref{deff}. The involution $f\mapsto g=(1-f)/(1+f)$ gives
$$
\biggl(\frac{2g}{1-g^2}\biggr)^2=\biggl(\frac{1-f^2}{2f}\biggr)^2
$$
and the modularity of the eta quotient gives
$$
\psi\biggl(\frac{-1}{2\tau}\biggr)=\frac{1}{\psi(\tau)}.
$$
Hence these involutions coincide and mere algebra shows that $R(g)=-R(f)$.
\end{proof}

It is instructive to consider Klein's $j$-invariant
$$
j(\tau) = \frac{64(1+4\psi(\tau))^3}{\psi(\tau)},
$$
for which we know that
$$
j(i)=j\biggl(\frac{1+i}{2}\biggr)=12^3, \quad
j\biggl(\frac{i}{\sqrt2}\biggr)=j(\sqrt2i)=j\biggl(\frac{1+\sqrt2i}{3}\biggr)=20^3.
$$
By subtraction, we obtain a factorisation of
$$
j(\tau)-j(i)=\frac{64(1+\psi(\tau))(1-8\psi(\tau))^2}{\psi(\tau)}
$$
from which we conclude that
$$
\psi(i)=\frac18, \quad \psi\biggl(\frac{1+i}{2}\biggr)=-1.
$$
Similarly, from
$$
j(\tau)-j(\sqrt2i)=\frac{64(1-\psi(\tau))}{\psi(\tau)}\,\bigl(1-112\psi(\tau)-64\psi(\tau)^2\bigr)
$$
we obtain
$$
\psi\biggl(\frac{i}{\sqrt2}\biggr)=1, \quad
\psi(\sqrt2i)=\biggl(\frac{\sqrt2-1}{2}\biggr)^3, \quad
\psi\biggl(\frac{1+\sqrt2i}{3}\biggr)=-\biggl(\frac{\sqrt2+1}{2}\biggr)^3.
$$

We know of seven cases with positive $\psi(\tau)\in\mathbb Q[\sqrt2\,]$, namely
$\psi_k=\psi(2^{(k-1)/2}i)$ with integer $k\in[-3,3]$. The corresponding
values of $f$ are given by
$$
f_k=\frac{\sqrt{1+\psi_k}-1}{\sqrt{\psi_k}}.
$$
The involution gives $\psi_{-k}=1/\psi_k$ and hence $f_{-k}=(1-f_k)/(1+f_k)$.
At the fixed point, we have $\psi_0=1$ and hence $f_0=\sqrt{2}-1$.
The remaining cases are
$$
\psi_1=\frac18, \quad \psi_2=\frac{f_0^3}8, \quad \psi_3=\sqrt8\psi_2^2,
$$
with $\psi_1$ from $j(i)=12^3$ and $\psi_2$ from $j(\sqrt2i)=20^3$. We used $j(2i)=66^3$ to factorise
$$
j(\tau)-j(2i)=\frac{8(8-\psi(\tau))}{\psi(\tau)}\bigl(1-4480\psi(\tau)-512\psi(\tau)^2\bigr)
$$
with $\psi_3$ given by the positive root of the quadratic factor.
Thus we obtain
$$
\phi(i)=\frac{R(f_1)}{2^{21/2}\pi^6}\Bigl(\Gamma\Bigl(\frac14\Bigr)\Bigr)^8=-\frac14\,\phi\Bigl(\frac{i}{2}\Bigr), \quad
\phi(2i)=\frac{f_0R(f_3)}{2^{55/4}\pi^6}\Bigl(\Gamma\Bigl(\frac14\Bigr)\Bigr)^8=-\frac{1}{64}\,\phi\Bigl(\frac{i}{4}\Bigr)
$$
from the Chowla--Selberg formula at the first singular value. At the second, we have
$$
\phi\biggl(\frac{i}{\sqrt2}\biggr)=0, \quad
\phi(\sqrt2i)=\frac{f_0^{5/2}R(f_2)}{2^{21/2}\pi^2}\biggl(\frac{\Gamma(\frac18)}{\Gamma(\frac58)}\biggr)^4
=-\frac{1}{16}\,\phi\biggl(\frac{i}{2\sqrt2}\biggr).
$$
Moreover, we obtain values for $R(f_k)\in\mathbb Q[\sqrt2\,]$, as follows:
$$
R(f_1)=-\frac{7}{3^2}, \quad R(f_2)=-\frac{80f_0+15}{7^2}, \quad R(f_3)=-\frac{352f_0+295}{21^2}.
$$

Conjecture~\ref{conj1} means that the anti-derivative of $\phi(\tau)$ has a Fourier expansion with integer coefficients,
a somewhat unusual property for a modular form (of weight~4).
This is possible only because $\phi(\tau)$ has poles in the upper-half of the complex plane. The dominant
singularity is at $\tau=(1+i)/2$, where $\psi(\tau)=-1$ and hence $f^2=-1$. Thus,
$$
R(f)=2\biggl(\frac{2f}{1+f^2}\biggr)^2-1=-R\biggl(\frac{1-f}{1+f}\biggr)
$$
has a double pole at $\tau=(1+i)/2$. This led us to expect an exponential growth
of the form $A(n)=(-1)^n(C+O(1/n))e^{n\pi}$ at large $n$, for some positive constant $C$.
Empirically, we discovered that $C=\frac12e^{\pi/2}$ and even more remarkably that
$$
2(-1)^nA(n) = E\biggl(\frac{(2n+1)\pi}{2}\biggr) + O(e^{n\pi/5})
$$
with a very simple exponentially increasing function
$$
E(x)=\frac{x-1}{x}\,e^x.
$$
Thus the leading term gives 80\% of the decimal digits of $A(n)$. Intrigued by this, we studied
the next to leading term, finding
$$
2(-1)^nA(n) - E\biggl(\frac{(2n+1)\pi}{2}\biggr)
=2\cos\biggl(\frac{2(n-2)\pi}{5}\biggr)E\biggl(\frac{(2n+1)\pi}{10}\biggr)+O(e^{n\pi/13}),
$$
which determines more than 92\% of the digits. This led us to suppose that
\begin{equation}
2(-1)^nA(n) = \sum_{m=1}^M C(m,n)E\biggl(\frac{(2n+1)\pi}{2m}\biggr) + o(e^{n\pi/M})
\label{ansatz}
\end{equation}
with trigonometrical coefficients $C(m,n)$ that are non-zero if and only if $m$ belongs to
a sequence $\bS$ of integers, beginning with $1, 5, 13$.
Further numerical work revealed that this sequence continues as follows:
$$
1, \; 5, \; 13, \; 17, \; 25, \; 29, \; 37, \; 41, \; 53, \; 61, \; 65, \; 73, \; 85, \; 89, \; 97, \; \dots,
$$
from which we inferred that $\bS$ is the sequence of positive integers
divisible only by primes congruent to 1 modulo~4.

We have $C(1,n)=1$, while for $m\in\bS$ and $1<m<65$ we found that
$C(m,n)=2$ when $m\mid(2n+1)$. However $C(65,n)=4$ when $65\mid(2n+1)$,
which suggests that the coefficients depend upon the factorisation of $m$ into powers of distinct primes.

\begin{conjecture}
\label{conj2}
If $m>1$ is a product of $\omega$ powers $q_k=p_k^{e_k}$
with distinct primes $p_k\equiv1\bmod4$, there is a positive integer $r(m)<m/2$ such that \eqref{ansatz} holds with
\begin{equation}
C(m,n)=2^\omega\prod_{k=1}^\omega\cos\biggl((2n+1-m)\frac{r(m)\pi}{q_k}\biggr)
\qquad\text{for}\quad
m=\prod_{k=1}^\omega q_k\,.
\label{conj}
\end{equation}
Furthermore, $C(m,n)=0$ for $m>1$ not of the latter form.
\end{conjecture}

For $\omega=1$ and $m\in\bS$, a unique positive integer $r(m)<m/2$
specifies $C(m,n)$. For $\omega>1$, there are $2^{\omega-1}$ such integers, of which
we select the smallest. For example $r(65)=7$ tells us that
\begin{align*}
C(65,n)&=4\cos\biggl((2n-64)\frac{7\pi}{5}\biggr)\cos\biggl((2n-64)\frac{7\pi}{13}\biggr)
\\
&=4\cos\biggl(2(n-2)\frac{2\pi}{5}\biggr)\cos\biggl(2(n-6)\frac{6\pi}{13}\biggr),
\end{align*}
where the second line comes from properties of the cosines. Any value of $r(65)$ with
residues $\pm2\bmod5$ and $\pm 6\bmod13$ gives the same result.

Continuing this empirical work, we determined the
asymptotic expansion~\eqref{ansatz} up to $m=5\times13\times17=1105$,
using least-squares fitting of exact integer
data for $A(n)$ with $n\in[30000,\,32200]$, obtaining for example  $r(1105)=216$ and hence
$$
C(1105,n)=8\cos\biggl(2(n-2)\frac{\pi}{5}\biggr)\cos\biggl(2(n-6)\frac{5\pi}{13}\biggr)\cos\biggl(2(n-8)\frac{5\pi}{17}\biggr).
$$
This fit gave 99.9\% of the digits of $A(n)$. We continued up to $m=2017$ and obtained
$$
C(2017,n)=2\cos\biggl(2(n-1008)\frac{894\pi}{2017}\biggr);
$$
this fit gave 99.95\% of the digits of $A(n)$.
Then we discovered an algebraic method.

To find $r(m)$, we may study the orbit of $(1+i)/2$ in $\Gamma_0(2)$, which is the
group of M\"obius transformations $M(a,b,c,d)$ with integer arguments such that
$ad-bc=1$ and, most crucially, $c\equiv0\bmod2$. The action is specified by
$\tau\mapsto(a\tau+b)/(c\tau+d)$.
Since $c$ is even, $a$ and $d$ must be odd. Acting on the singularity of $\phi(\tau)$ at
$\tau=(1+i)/2$, the transformation $M(a,b,c,d)$ locates another singularity at
$$
\tau=\frac{a(1+i)+2b}{c(1+i)+2d}=\frac{w+i}{2m}
$$
with integers
$$
m=(c/2)^2+(c/2+d)^2,\quad w=(a+b)(c+d)+bd.
$$

There are $2^{\omega-1}$ essentially different ways
of expressing an integer $m\in\bS$ as a sum of coprime squares.
Suppose that we take one of these, say $m=\gamma^2+\delta^2$
with positive odd $\gamma$. Then we set $c=2\gamma$
and $d=\delta-\gamma$. For $a$, we take the unique positive odd
integer with $a<c$ and $c\mid(ad-1)$. Then we set $b=(ad-1)/c$
and compute $w$ by the formula above. For precisely one of the
two choices of sign for $\delta$ we obtain an odd positive integer $w<m$.
If $\omega=1$, we finish by setting $r(m)=(m-w)/2$.

More generally, we have $\mu=2^{\omega-1}$ singularities at $\tau=(w_j+i)/(2m)$. For each of these
we define $r_j=(m-w_j)/2$. The conjecture then requires that
$$
\mu\prod_{k=1}^\omega\cos\biggl((2n+1-m)\frac{r(m)\pi}{q_k}\biggr)
=\sum_{j=1}^{\mu}\cos\biggl((2n+1-m)\frac{r_j\pi}{m}\biggr)
$$
for at least one $r(m)<m/2$. Repeated use of
$2\cos(\alpha)\cos(\beta)=\cos(\alpha+\beta)+\cos(\alpha-\beta)$
enabled us to verify this claim for all $m\in\bS$ with $m\in[5,160225]$.
At $m=5^2\times13\times17\times29$, we found that $r(160225)=2999$
reproduces the singularities located by $M(a,b,c,d)$ with these eight sets of arguments:
\begin{gather*}
(5, -11, 226, -497), \; (39, -32, 674, -553), \;
(55, -81, 366, -539), \\ (25, -73, 162, -473), \;
(181, -73, 786, -317), \; (187, -101, 798, -431), \\
(215, -182, 658, -557), \; (253, -24, 622, -59).
\end{gather*}

\section({Laurent expansion of \003\306(\003\304) near a double pole})%
{Laurent expansion of $\phi(\tau)$ near a double pole}
\label{sec:phi-laurent}

Using $R(f)=(\psi(\tau)-1)/(\psi(\tau)+1)$, we may develop the Laurent expansion of
$$
\phi(\tau) = \frac{(\eta_1\eta_2)^4}{2}\left(\frac{64\eta_2^{24}-\eta_1^{24}}{64\eta_2^{24}+\eta_1^{24}}\right)
$$
near $\tau=(1+i)/2$ by purely algebraic methods, as follows.

All derivatives of $\log(\eta_1)$ with respect to $\log(q)=2\pi i\tau$ are polynomials in the Eisenstein series
$$
L(q)=1-24\sum_{n=1}^\infty\frac{nq^n}{1-q^n},\quad
M(q)=1+240\sum_{n=1}^\infty\frac{n^3q^n}{1-q^n},\quad
N(q)=1-504\sum_{n=1}^\infty\frac{n^5q^n}{1-q^n}.
$$
To determine such polynomials, we use Ramanujan's convenient formulas
$$
\frac{q}{\eta_1}\frac{\d\eta_1}{\d q}=\frac{L}{24},\quad
q\frac{\d L}{\d q}=\frac{L^2-M}{12},\quad
q\frac{\d M}{\d q}=\frac{LM-N}{3},\quad
q\frac{\d N}{\d q}=\frac{LN-M^2}{2}.
$$
Similarly, derivatives of $\log(\eta_2)$ are polynomials in $L(q^2)$, $M(q^2)$ and $N(q^2)$.
At $\tau=(1+i)/2$, with $q=q_0=-\exp(-\pi)$, we have the evaluations
$$
L(q_0)=2L(q_0^2)=\frac6\pi, \quad
M(q_0)=-4M(q_0^2)=-\frac{3(\Gamma(\frac14))^8}{16\pi^6}, \quad
N(q_0)=N(q_0^2)=0,
$$
thanks to exposition \cite{BB87} by our departed friend Jon Borwein and his brother Peter.

Next, we convert the Taylor series for $\log(\eta_1)$ and $\log(\eta_2)$ near $\tau=(1+i)/2$
to Taylor series for $(\eta_1\eta_2)^4$ and $(\eta_2/\eta_1)^{24}$, using initial values
$$
(\eta_1\eta_2)^4=-\frac{i}{24}M(q_0), \quad
\biggl(\frac{\eta_2}{\eta_1}\biggr)^{24}=-\frac{1}{64}
\qquad \text{at} \quad \tau=\frac{1+i}{2}.
$$
Thus we are able to prove, by mere algebra, that
$$
\phi(\tau)=\frac{-i}{8\pi^2(\tau-\tau_+)^2(\tau-\tau_-)^2}
\biggl(1+\frac{56D}{15}\biggl(\frac{\tau-\tau_+}{\tau-\tau_-}\biggr)^4+O\bigl((\tau-\tau_+)^8\bigr)\biggr)
$$
with $\tau_{\pm}=(1\pm i)/2$ and a rapidly computable AGM in
$$
D=\biggl(\frac{\pi^2M(q_0)}{24}\biggr)^2
=\frac{(\Gamma(\frac14))^{16}}{2^{14}\pi^8}
=\frac{\pi^4}{4(\agm(1,\sqrt2))^8}.
$$

After more algebraic investigation, we were led to posit that
\begin{equation}
\phi(\tau) = \frac{-i}{8\pi^2(\tau-\tau_+)^2(\tau-\tau_-)^2}
\biggl(1-2\sum_{n=1}^\infty c_n\biggl(-4D\frac{(\tau-\tau_+)^4}{(\tau-\tau_-)^4}\biggr)^n\biggr)
\label{laurent}
\end{equation}
with rational coefficients $c_n$. Confirmation came from the determinations
\begin{gather*}
c_1=\tfrac {7}{15}, \;
c_2=\tfrac{57}{175}, \;
c_3=\tfrac{47953}{482625}, \;
c_4=\tfrac{28647821}{1206079875}, \;
c_5=\tfrac{21064211}{3897196875}, \\
c_6=\tfrac{140089261833377}{118706391513084375}, \;
c_7=\tfrac{7572730553099}{30813510149296875}, \;
c_8=\tfrac{7162997611208195563}{144310550800696358203125},
\end{gather*}
which may be proved by developing 36 terms of the Laurent expansion.

Thereafter, algebra becomes tedious. However, a numerical method
is efficient, thanks to fast evaluation of eta values by \texttt{Pari-GP},
which took merely a minute to evaluate 420\,000 digits of $D$ and $\phi(\tau_{+}+10^{-510}i)$.
Then rational values of $c_n$ for $n\in[1,200]$ followed in seconds.
This data determines $804$ terms in the Laurent expansion.
The following properties were observed:
\begin{itemize}
\item[(a)] $c_n=a_n/b_n$ is a positive ratio of odd integers;
\item[(b)] no prime greater than $4n+1$ divides $b_n$;
\item[(c)] the denominator of $(4n+5)!\,c_n$ is square free and divisible only by primes $p<n$ with $p\equiv1\bmod 4$.
\end{itemize}
At large~$n$,
$$
c_n=\frac{8n-6}{D^n}\bigl(1+O((5/8)^{2n})\bigr)
$$
with a leading term from the singularity at $\tau=(3+i)/10$ and an
oscillating next to leading term from the singularity at $\tau=(5+i)/26$. We used the sum rule
$$
2\sum_{n=1}^\infty c_nD^n(2-\sqrt2)^{4n}=1
$$
to test our results. This follows from the vanishing of \eqref{laurent} at $\tau=i/\sqrt2$.
Using exact values of $c_n$ for $n\leq200$ and thereafter the approximation $c_nD^n\approx8n-6$,
we obtained a left hand side differing from unity by less than $10^{-265}$.

\section{The inhomogeneous differential equation}
\label{sec:DE}

In this section we prove Theorem~\ref{th1}.

Write the double integral \eqref{dint} as
$$
I_2(f)=\frac1{1+f}\iint_{\pi/2>\alpha>\beta>0}\frac{\d\alpha\,\d\beta}
{\sqrt{(1-h(f)\cos^2\alpha)(1-h(f)\sin^2\beta)}},
$$
where $h(f)=4f/(1+f)^2=h(1/f)$,
and perform the change
$$
\sin\beta=u^{1/2}, \qquad \sin\alpha=v^{1/2},
$$
so that
$$
4I_2(f)
=\frac1{1+f}\iint_{0<u<v<1}\frac{\d u\,\d v}
{\sqrt{uv(1-u)(1-v)(1-uh(f))(1-(1-v)h(f))}}.
$$

Consider the series
\begin{align}
Y(h)
&=\iint_{0<u<v<1}\frac{\d u\,\d v}{\sqrt{uv(1-u)(1-v)(1-uh)(1-(1-v)h)}}
\nonumber\\
&=\sum_{m,n=0}^\infty\frac{(\frac12)_m(\frac12)_n}{m!\,n!}\,h^{m+n}\iint_{0<u<v<1}u^{m-1/2}(1-u)^{-1/2}v^{-1/2}(1-v)^{n-1/2}\d u\,\d v
\nonumber\displaybreak[2]\\
&=\sum_{m,n=0}^\infty\frac{(\frac12)_m(\frac12)_n}{m!\,n!}\,h^{m+n}
\times\frac{m!\,(\frac12)_n}{(m+\frac12)\,(\frac12)_{m+n+1}}\,
{}_3F_2\biggl(\begin{matrix} \frac12, \, m+\frac12, \, m+1 \\[1mm] m+\frac32, \, m+n+\frac32 \end{matrix} \biggm|1\biggr)
\nonumber\displaybreak[2]\\
&=\sum_{m,n=0}^\infty\frac{(\frac12)_m(\frac12)_n}{m!\,n!}\,h^{m+n}
\times\frac{m!\,n!}{(m+n)!\,(m+\frac12)\,(n+\frac12)}\,
{}_3F_2\biggl(\begin{matrix} 1, \, \frac12, \, \frac12 \\[1mm] m+\frac32, \, n+\frac32 \end{matrix} \biggm|1\biggr)
\nonumber\displaybreak[2]\\
&=\sum_{m,n=0}^\infty\frac{(\frac12)_m(\frac12)_n}{(m+n)!}\,h^{m+n}
\sum_{k=0}^\infty\frac{(\frac12)_k^2}{(m+\frac12)_{k+1}(n+\frac12)_{k+1}},
\label{eq1}
\end{align}
where the standard hypergeometric notation is used \cite{Ba35} and a transformation due to Thomae \cite[Sect.~3.2]{Ba35} is applied.

The representation \eqref{eq1} can be combined with the classical Gosper--Zeilberger algorithm to produce a differential equation for $Y(h)$.

\begin{theorem}
\label{th2}
In a neighbourhood of $h=0$ the function $Y(h)$ satisfies
\begin{equation}
\bigl(\theta_h^3-\tfrac12h(2\theta_h+1)(2\theta_h^2+2\theta_h+1)+h^2(\theta_h+1)^3\bigr)Y(h)
=-\frac{(4-4h-h^2)h}{(2-h)^2\sqrt{1-h}},
\label{DE}
\end{equation}
where $\theta_h=h\,\frac{\d}{\d h}$.
\end{theorem}

\begin{proof}
The coefficient of $h^n$ in $Y(h)$ is equal to
$$
S_0(n)=\sum_{m=0}^n\sum_{k=0}^\infty a(n;m,k),
$$
where
$$
a(n;m,k)=\frac{(\frac12)_m(\frac12)_{n-m}(\frac12)_k^2}{n!\,(m+\frac12)_{k+1}(n-m+\frac12)_{k+1}}.
$$

First, observe that
\begin{align*}
&
(n-m+1)^2a(n+1;m,k)
-(n-m+\tfrac12)^2a(n;m,k)
\\ &\quad
=\frac{(n-m+\tfrac12)^2}{n+1}\bigl(\tilde a(n;m,k+1)-\tilde a(n;m,k)\bigr),
\end{align*}
where $\tilde a(n;m,k)=(m+k+\frac12)a(n;m,k)$. Since $\tilde a(n;m,k)\to0$ as $k\to\infty$, summing both sides for $k=0,1,\dots$ we deduce that
\begin{align*}
&
(n-m+1)^2A(n+1;m)-(n-m+\tfrac12)^2A(n;m)
\\ &\quad
=-\frac{(n-m+\tfrac12)^2}{n+1}\,\tilde a(n;m,0)
=-\frac{(\frac12)_m(\frac12)_{n+1-m}}{(n+1)!}
\end{align*}
in the notation $A(n;m)=\sum_{k=0}^\infty a(n;m,k)$.
Using the symmetry $a(n;m,k)=a(n;n-m,k)$, so that $A(n;m)=A(n;n-m)$, we can write the difference equation as
\begin{equation}
m^2A(n+1;m)-(m-\tfrac12)^2A(n;m-1)
=-\frac{(\frac12)_{n+1-m}(\frac12)_m}{(n+1)!}.
\label{R3}
\end{equation}

Denote
\begin{gather*}
S_0(n)=\sum_{m=0}^nA(n;m), \quad
S_1(n)=\sum_{m=0}^nmA(n;m), \quad
S_2(n)=\sum_{m=0}^nm^2A(n;m)
\\
\text{and}\quad
S_3(n)=\sum_{m=0}^nm^3A(n;m).
\end{gather*}
In view of the symmetry
$$
S_1(n)=\sum_{m=0}^n(n-m)A(n;m)=nS_0(n)-S_1(n)
$$
and
$$
S_3(n)=\sum_{m=0}^n(n-m)^3A(n;m)=n^3S_0(n)-3n^2S_1(n)+3nS_2(n)-S_3(n),
$$
so that
$$
S_1(n)=\tfrac12nS_0(n)
\quad\text{and}\quad
S_3(n)=\tfrac32nS_2(n)-\tfrac14n^3S_0(n).
$$
Furthermore,
\begin{align*}
\sum_{m=1}^{n+1}(m-\tfrac12)^2A(n;m-1)
&=\sum_{m=0}^n(m+\tfrac12)^2A(n;m)
\\
&=S_2(n)+\tfrac14(2n+1)S_0(n).
\end{align*}
Summing \eqref{R3} over $m=1,\dots,n+1$ we obtain
\begin{equation}
S_2(n+1)-S_2(n)-\tfrac14(2n+1)S_0(n)
=-T(n+1),
\qquad\text{where}\quad
T(n)=\sum_{m=1}^n\frac{(\frac12)_{n-m}(\frac12)_m}{n!}.
\label{S-}
\end{equation}

We now multiply \eqref{R3} by $n+1$ and its shift
\begin{equation*}
m^2A(n;m)-(m-\tfrac12)^2A(n-1;m-1)
=-\frac{(\frac12)_{n-m}(\frac12)_m}{n!}
\end{equation*}
by $n-m+\frac12$ and subtract the latter from the former:
\begin{align*}
&
(n+1)m^2A(n+1;m)-m^2(n-m+\tfrac12)A(n;m)
\\ &\quad
-(n+1)(m-\tfrac12)^2A(n;m-1)+(m-\tfrac12)^2(n-m+\tfrac12)A(n-1;m-1)
=0.
\end{align*}
Summing this over $m=1,\dots,n$ we deduce that
\begin{align*}
&
(n+1)\sum_{m=0}^{n+1}m^2A(n+1;m)-\sum_{m=0}^nm^2(n-m+\tfrac12)A(n;m)
\\ &\;\quad
-(n+1)\sum_{m=0}^n(m+\tfrac12)^2A(n;m)+\sum_{m=0}^{n-1}(m+\tfrac12)^2(n-m-\tfrac12)A(n-1;m)
\\ &\;
=(n+1)\bigl((n+1)^2A(n+1;n+1)-(n+\tfrac12)^2A(n;n)\bigr)
=-\frac{(\frac12)_{n+1}}{n!}
\end{align*}
implying
\begin{align*}
&
(n+1)S_2(n+1)+S_3(n)-(n+\tfrac12)S_2(n)-(n+1)(S_2(n)+S_1(n)+\tfrac14S_0(n))
\\ &\;\quad
-S_3(n-1)+(n-\tfrac32)S_2(n-1)+(n-\tfrac34)S_1(n-1)+\tfrac18(2n-1)S_0(n-1)
\\ &\;
=-\frac{(\frac12)_{n+1}}{n!}
\end{align*}
that after reductions becomes
\begin{align*}
&
(n+1)S_2(n+1)
-\tfrac12(n+3)S_2(n)
-\tfrac14(n^3+2n^2+3n+1)S_0(n)
\\ &\;\quad
-\tfrac12nS_2(n-1)
+\tfrac18n(2n^2-2n+1)S_0(n-1)
=-\frac{(\frac12)_{n+1}}{n!},
\end{align*}
which is in turn
\begin{align*}
&
(n+1)(\tfrac14(2n+1)S_0(n)-T(n+1))
+\tfrac12n(\tfrac14(2n-1)S_0(n-1)-T(n))
-\tfrac12S_2(n)
\\ &\;\quad
-\tfrac14(n^3+2n^2+3n+1)S_0(n)
+\tfrac18n(2n^2-2n+1)S_0(n-1)
=-\frac{(\frac12)_{n+1}}{n!},
\end{align*}
or
$$
-\tfrac12S_2(n)
-\tfrac14n^3S_0(n)
+\tfrac14n^3S_0(n-1)
=(n+1)T(n+1)+\tfrac12nT(n)-\frac{(\frac12)_{n+1}}{n!}.
$$
Shifting the resulting equation,
\begin{align*}
&
-\tfrac12S_2(n+1)
-\tfrac14(n+1)^3S_0(n+1)
+\tfrac14(n+1)^3S_0(n)
\\ &\;
=(n+2)T(n+2)+\tfrac12(n+1)T(n+1)-\frac{(\frac12)_{n+2}}{(n+1)!},
\end{align*}
and subtracting the latter from the former we obtain
\begin{align*}
&
\tfrac12(S_2(n+1)-S_2(n))
+\tfrac14(n+1)^3S_0(n+1)
-\tfrac14((n+1)^3+n^3)S_0(n)
+\tfrac14n^3S_0(n-1)
\\ &\;
=-(n+2)T(n+2)+\tfrac12(n+1)T(n+1)+\tfrac12nT(n)+\frac{(\frac12)_{n+1}}{2(n+1)!}.
\end{align*}
Applying \eqref{S-} this finally leads to
\begin{align}
&
\tfrac14(n+1)^3S_0(n+1)
-\tfrac18(2n+1)(2n^2+2n+1)S_0(n)
+\tfrac14n^3S_0(n-1)
\nonumber\\ &\;
=-(n+2)T(n+2)+\tfrac12(n+2)T(n+1)+\tfrac12nT(n)+\frac{(\frac12)_{n+1}}{2(n+1)!},
\label{IR}
\end{align}
which does not involve sums $S_2(n)$.

Notice that $T(n)=\sum_{m=1}^nc(n;m)$, where
$$
c(n;m)=\frac{(\frac12)_m(\frac12)_{n-m}}{n!}=\frac{(-1)^m(\frac12-m)_n}{n!}
$$
satisfies
$$
2c(n+1;m)-c(n;m)=c(n+1;m)-c(n+1;m+1).
$$
Summing the equality over $m=1,\dots,n$ we have
\begin{align*}
2T(n+1)-T(n)
&=2c(n+1;n+1)+(c(n+1;1)-c(n+1;n+1))
=\frac{(\frac12)_n}{n!}.
\end{align*}
The recursion immediately implies that
$$
\sum_{n=0}^\infty T(n)h^n=\frac{h}{(2-h)\sqrt{1-h}}
$$
and also simplifies the right-hand side of \eqref{IR}:
\begin{align*}
&
\tfrac14(n+1)^3S_0(n+1)
-\tfrac18(2n+1)(2n^2+2n+1)S_0(n)
+\tfrac14n^3S_0(n-1)
\\ &\;
=\tfrac12nT(n)-\frac{(\frac12)_{n+1}}{2n!}.
\end{align*}
It remains to notice that
\begin{align*}
\sum_{n=0}^\infty\biggl(nT(n)-\frac{(\frac12)_{n+1}}{n!}\biggr)h^n
&=h\,\frac{\d}{\d h}\biggl(\frac{h}{(2-h)\sqrt{1-h}}\biggr)-\frac1{2(1-h)\sqrt{1-h}}
\\
&=\frac{h(4-2h-h^2)}{2(2-h)^2(1-h)\sqrt{1-h}}-\frac1{2(1-h)\sqrt{1-h}}
\\
&=-\frac{4-4h-h^2}{2(2-h)^2\sqrt{1-h}}
\end{align*}
to translate the resulting difference equation into the differential equation \eqref{DE}.
\end{proof}

It is clear that the singularities of the differential equation \eqref{DE} are at the points $h=0,1,2,\infty$,
where the singularity at $h=2$ occurs because of the inhomogeneity.

\begin{proof}[Proof of Theorem~\textup{\ref{th1}}]
We can translate the differential equation of Theorem~\ref{th2} into one for the function
$$
I(f)=I_2\biggl(\frac{1-f}{1+f}\biggr),
$$
using the defining expression
\begin{align*}
\frac{I(f)}{1+f}
&=\frac12\iint_{0<u<v<1}\frac{\d u\,\d v}
{\sqrt{uv(1-u)(1-v)(1-u(1-f^2))(1-(1-v)(1-f^2))}}
\\
&=\frac12Y(1-f^2).
\end{align*}
The result is
$$
\frac1f\bigl(\theta_f^3-2f(\theta_f^3+\theta_f)f+f^2\theta_f^3f^2\bigr)\frac{I(f)}{1+f}
=2\biggl(\frac{2f}{1+f^2}\biggr)^2-1.
$$
It follows then from Lemma~\ref{lem1} that
$$
I(f)=I\biggl(\frac{1-f}{1+f}\biggr),
$$
hence $I(f)=I_2(f)$.
\end{proof}

In summary, our proof exploits a power-series development of \eqref{dint} in a suitable base
(namely, with respect to $h=4f/(1+f)^2$) and its algorithmic fitness for a telescoping argument.
This structure suggests the existence of double integrals that are similar to or more general than
Ausserlechner's integral, which satisfy arithmetic inhomogeneous differential equations and have
transformation properties of type~\eqref{udo}.

\section{Final remarks}
\label{sec:FR}

A different, two-variable generalization of the complete elliptic integral was considered by W.\,N.~Bailey in \cite{Ba48}.

A general form of the Taylor expansion of a modular form near a CM (elliptic) point (that is, a quadratic irrationality from the upper-half of the complex plane)
is analysed in \cite{Wo84}. The principal theorem there can be used to explain the particular form of Laurent expansion in \eqref{laurent}.

The differential equation for \eqref{dint} reveals an inhomogeneous term \eqref{rel} with the three symmetries
$$
R(f) = R(-f) = R\biggl(\frac1f\biggr) = -R\biggl(\frac{1-f}{1+f}\biggr).
$$
The last of these is derivable from, yet obscured by, the expansion in $h = 4f/(1+f)^2$.

Finally, we point out again the mysteries behind a modular parametrisation of the double integral: our unexpected discovery of the integrality
of the sequence $A(0),A(1),A(2),\dots$ in Section~\ref{sec:NE} (Conjecture~\ref{conj1}) and an arithmetic pattern of
the asymptotic growth for the sequence in Section~\ref{sec:phi} (Conjecture~\ref{conj2}).

\begin{acknowledgements}
We thank Duco van Straten for his interest and effort to prove Theorem~\ref{th1}, and the anonymous referee for an enthusiastic report.
The second author thanks the Max Planck Institute for Mathematics in Bonn for research support and providing excellent
working conditions during his stay in July--August 2017.
\end{acknowledgements}

\end{document}